\documentclass[12pt]{amsart}

\usepackage{amsmath,amsthm,amssymb,amsfonts}
\usepackage{mathtools}
\usepackage{enumitem}
\usepackage[colorlinks=true,linkcolor=black,citecolor=black,urlcolor=blue]{hyperref}

\theoremstyle{plain}
\newtheorem{theorem}{Theorem}[section]
\newtheorem{lemma}[theorem]{Lemma}
\newtheorem{corollary}[theorem]{Corollary}
\newtheorem{proposition}[theorem]{Proposition}

\theoremstyle{definition}
\newtheorem{definition}[theorem]{Definition}

\newtheorem{example}[theorem]{Example}

\theoremstyle{remark}
\newtheorem{remark}[theorem]{Remark}

\numberwithin{equation}{section}

\newcommand{\Fp}{\mathbb{F}_p}
\newcommand{\Ftwo}{\mathbb{F}_2}

\newcommand{\res}{\operatorname{res}}
\newcommand{\im}{\operatorname{im}}
\newcommand{\Span}{\operatorname{span}}

\title[Fibre-Product Stability and Composita]{Fibre-Product Stability, Kummer Gluing, and Composita for Universal Koszulity}
\author{Marina Palaisti}
\date{\today}

\begin{document}

	\begin{abstract}
		We study universal Koszulity under algebraic gluing and field composita.
		For quadratic fibre products, orthogonally split pullbacks reduce to direct
		sums, while Stanley--Reisner pullbacks arise from graph gluing preserving
		universal Koszulity under a dominating-clique hypothesis. For composita,
		Kummer theory and the norm-residue theorem yield a graph criterion from
		restriction images and degree-two cup-product relations, verified for
		iterated Laurent-series fields with nonzero mixed symbols. In a nonabelian
		RAAG family, every finite \(p\)-extension has explicitly determined
		universally Koszul cohomology, with
		\[
		G_L(p)\cong F_{1+q(d-1)}\times\mathbb Z_p^s.
		\]
	\end{abstract}

\maketitle



\section{Introduction}\label{sec:intro}

Universal Koszulity is a strengthening of Koszulity for quadratic algebras
generated in degree one. In Galois cohomology it is formulated by requiring
every left ideal generated in degree one to have a linear resolution.
Enhanced Koszul properties in Galois cohomology were developed in
\cite{MinacPalaistiPasiniTan,PalaistiThesis}; related Koszul, PBW, and
quadratic-dual phenomena are studied in
\cite{MinacPasiniQuadrelliTan}. We use in particular the direct-sum
stability theorem \cite[Proposition~30]{MinacPalaistiPasiniTan}.

This paper develops two forms of gluing. The first is algebraic. Given
surjective graded algebra maps
\[
\varphi_i:A_i\twoheadrightarrow C,\qquad i=1,2,
\]
we study the fibre product
\[
A=A_1\times_C A_2.
\]
The behaviour of degree-one-generated ideals under this construction is
controlled by their component ideals and colon ideals. Section~\ref{sec:generalcriterion}
records the exact colon formula and a general criterion using
\[
\widehat{\mathcal L}(A)=\mathcal L(A)\cup\{A\},
\]
which accounts for the case in which a component colon equals the whole
algebra.

The main algebraic results occur in Section~\ref{sec:structural}. For an
orthogonally split map
\[
A_i\cong C\sqcap B_i\longrightarrow C,
\]
the pullback has the explicit form
\[
A_1\times_C A_2\cong C\sqcap B_1\sqcap B_2,
\]
so direct-sum stability gives universal Koszulity. For exterior
Stanley--Reisner algebras, a gluing of graphs along a common induced subgraph
realizes the corresponding algebraic fibre product. When the two graphs are
diagonal and the common subgraph is a dominating clique, the glued graph is
diagonal. Cassella--Quadrelli's characterization then yields universal
Koszulity \cite{CassellaQuadrelli}.

The second form of gluing is arithmetic. Let \(K_1,K_2\) be algebraic
subextensions of a fixed separable closure of \(k\), and put \(K=K_1K_2\).
Inside the absolute Galois group of \(k\),
\[
G_K=G_{K_1}\cap G_{K_2}.
\]
The induced cohomological maps are the restrictions
\[
H^\bullet(K_i,\mathbb F_p)\longrightarrow H^\bullet(K,\mathbb F_p).
\]
Assume \(\operatorname{char}(k)\neq p\) and \(\mu_p\subseteq k\).
Kummer theory identifies \(H^1(F,\mathbb F_p)\) with
\(F^\times/F^{\times p}\), and the norm-residue theorem identifies cup
products with Milnor symbols. Since mod-\(p\) Galois cohomology is quadratic,
the restriction images in degree one and the full kernel of the degree-two
cup-product map determine the resulting quadratic presentation.

Section~\ref{sec:composita} gives two criteria in these terms. The
orthogonal criterion covers zero mixed products and disjoint positive-degree
subalgebras. The Kummer graph criterion allows a common restriction image
and a prescribed mixed multiplication table. Its hypothesis identifies the
entire degree-two relation space with the relation space of an exterior
Stanley--Reisner algebra. A diagonal total graph therefore gives universally
Koszul cohomology.

We verify this criterion for an explicit infinite family. Let \(\kappa\)
be algebraically closed with \(\operatorname{char}(\kappa)\neq p\), and
work inside
\[
M=\kappa((u_1))\cdots((u_d)).
\]
For \(S\subseteq\{1,\dots,d\}\), the field \(K_S\) is obtained by using
\(u_j\) in the coordinates indexed by \(S\) and \(u_j^p\) in the remaining
coordinates. We prove
\[
K_SK_T=K_{S\cup T},\qquad K_S\cap K_T=K_{S\cap T},
\]
compute all degree-one restriction images, and show that the mixed symbols are
the nonzero exterior basis elements. The total Kummer graph is complete.

Section~\ref{sec:raag-composita} gives a nonabelian family. By
\cite[Theorem~1.2]{SnopceZalesskii}, diagonal pro-\(p\) RAAGs have
hereditary RAAG subgroup structure and admit maximal pro-\(p\) Galois
realizations. Together with Cassella--Quadrelli, these properties give
stability throughout the finite \(p\)-extension lattice of a realizing field. For
\[
G_k(p)\cong F_d\times\mathbb Z_p^s,
\]
we classify every open subgroup arising from a finite field extension and
obtain
\[
G_L(p)\cong F_{1+q(d-1)}\times\mathbb Z_p^s,
\qquad
q=[F_d:\operatorname{pr}_{F_d}G_L(p)].
\]
Independent Kummer characters on the free factor give composita with
\[
G_K(p)\cong F_{1+p^r(d-1)}\times\mathbb Z_p^s.
\]
The corresponding cohomology algebras are exterior Stanley--Reisner algebras
of diagonal join graphs.

The paper is organized accordingly: Sections~\ref{sec:prelim}--\ref{sec:structural}
develop the algebraic results; Section~\ref{sec:composita} treats Kummer
gluing and Laurent-series composita; Section~\ref{sec:raag-composita}
treats diagonal-RAAG finite \(p\)-extension lattices and explicit Kummer
composita; the final section records further extensions of the method.

\section{Preliminaries}\label{sec:prelim}

Throughout, \(p\) is a fixed prime and all graded algebras in the
purely algebraic sections are connected graded algebras over
\(\Fp\). For a connected graded algebra \(A\), write
\[
A_+=\bigoplus_{n\geq 1}A^n.
\]

\subsection{Quadratic algebras and universal Koszulity}

A graded algebra
\[
A=\bigoplus_{n\geq0}A^n
\]
is \emph{quadratic} if \(A^0=\Fp\), if \(A\) is generated by \(A^1\),
and if
\[
A\cong T(A^1)/\langle R\rangle
\]
for a subspace \(R\subseteq A^1\otimes A^1\).

\begin{definition}[Universal Koszulity
{\cite{ConcaUniversallyKoszul,MinacPalaistiPasiniTan}}]
\label{def:uk}
Let \(A\) be a quadratic algebra. Denote by \(\mathcal L(A)\) the
collection of graded \emph{left} ideals of \(A\) generated by
subspaces of \(A^1\). The algebra \(A\) is \emph{universally Koszul}
if every \(I\in\mathcal L(A)\) has a linear free resolution as a
graded left \(A\)-module.
\end{definition}

\begin{lemma}[{\cite[Proposition~20]{MinacPalaistiPasiniTan}}]
\label{lem:colon-criterion}
A quadratic algebra \(A\) is universally Koszul if and only if for
every \(I\in\mathcal L(A)\) and every \(x\in A^1\setminus I^1\), the
left colon ideal
\[
I:(x)=\{a\in A: ax\in I\}
\]
belongs to \(\mathcal L(A)\).
\end{lemma}

\begin{remark}
For a graded-commutative algebra, an ideal generated by a degree-one
subspace is automatically two-sided. We nevertheless keep the
left-ideal formulation because it is the formulation used in
\cite{MinacPalaistiPasiniTan} and because the fibre-product arguments
do not require commutativity.
\end{remark}

For later use, define
\[
\widehat{\mathcal L}(A)=\mathcal L(A)\cup\{A\}.
\]
If \(I\in\mathcal L(A)\) and \(x\in A^1\), then universal Koszulity
implies
\[
I:(x)\in\widehat{\mathcal L}(A).
\]
Indeed, if \(x\notin I\), this is Lemma~\ref{lem:colon-criterion}, while
if \(x\in I\), then \(I:(x)=A\).

\subsection{Direct sums of graded algebras}

We recall the algebraic operation corresponding to free products of
pro-\(p\) groups in cohomology.

\begin{definition}[Direct-sum algebra]
\label{def:directsum}
Let \(A,B\) be connected graded \(\Fp\)-algebras. Their \emph{direct sum}
\[
A\sqcap B
\]
is the connected graded algebra whose underlying graded vector space is
\[
\Fp\oplus A_+\oplus B_+
\]
and whose multiplication restricts to the original multiplication on
\(A\) and on \(B\), while
\[
A_+B_+=B_+A_+=0.
\]
Equivalently,
\[
A\sqcap B\cong A\times_{\Fp}B
\]
with respect to the augmentation maps.
\end{definition}

\begin{theorem}[{\cite[Proposition~30]{MinacPalaistiPasiniTan}}]
\label{thm:directsum-uk}
If \(A\) and \(B\) are universally Koszul, then
\(A\sqcap B\) is universally Koszul.
\end{theorem}

By iteration, the direct sum of any finite family of universally
Koszul algebras is universally Koszul.

\subsection{Fibre products}

Let
\[
\varphi_i:A_i\twoheadrightarrow C,\qquad i=1,2,
\]
be surjective graded algebra maps. Their fibre product is
\[
A=A_1\times_C A_2
=
\{(a_1,a_2)\in A_1\oplus A_2:
\varphi_1(a_1)=\varphi_2(a_2)\}.
\]
It is a connected graded algebra under componentwise multiplication.

If \(I_i\subseteq A_i\) are left ideals and \(J\subseteq C\) is a left
ideal satisfying
\[
\varphi_i(I_i)=J,\qquad i=1,2,
\]
we write
\[
I_1\times_J I_2
=
\{(a_1,a_2)\in A:
a_i\in I_i\}.
\]

\section{A general fibre-product criterion}
\label{sec:generalcriterion}

The criterion in this section is intentionally formal. Its purpose
is to identify the precise colon compatibility needed in a completely
general fibre product. The structural results of
Section~\ref{sec:structural} will then replace these hypotheses by
concrete sufficient conditions.

\begin{lemma}
\label{lem:fibre-colon}
Let \(A=A_1\times_C A_2\), and let
\[
I=I_1\times_J I_2
\]
where \(I_i\subseteq A_i\) and \(J\subseteq C\) are left ideals with
\(\varphi_i(I_i)=J\). Let
\[
x=(x_1,x_2)\in A^1,\qquad
\bar x=\varphi_1(x_1)=\varphi_2(x_2)\in C^1.
\]
Then
\begin{equation}\label{eq:fibre-colon}
I:(x)
=
\{(a_1,a_2)\in A:
a_1\in I_1:(x_1),\
a_2\in I_2:(x_2)\}.
\end{equation}
Moreover, if
\[
L_i=I_i:(x_i),
\qquad
L_C=J:(\bar x),
\]
then
\[
\varphi_i(L_i)\subseteq L_C.
\]
If equality holds for both \(i=1,2\), then
\[
I:(x)=L_1\times_{L_C}L_2.
\]
\end{lemma}

\begin{proof}
Multiplication in \(A\) is componentwise:
\[
(a_1,a_2)(x_1,x_2)=(a_1x_1,a_2x_2).
\]
The product lies in \(I\) if and only if
\(a_ix_i\in I_i\) for \(i=1,2\), which proves
\eqref{eq:fibre-colon}.

If \(a_i\in L_i\), then \(a_ix_i\in I_i\), hence
\[
\varphi_i(a_i)\bar x
=
\varphi_i(a_ix_i)
\in
\varphi_i(I_i)
=
J.
\]
Thus \(\varphi_i(a_i)\in J:(\bar x)=L_C\), proving
\(\varphi_i(L_i)\subseteq L_C\). If the images are equal, then the
right-hand side of \eqref{eq:fibre-colon} is exactly the indicated
fibre product.
\end{proof}

\begin{proposition}
\label{prop:formal-criterion}
Let \(A=A_1\times_C A_2\), where \(A,A_1,A_2,C\) are quadratic
algebras generated in degree one. Assume:
\begin{enumerate}[label=\textnormal{(\roman*)}]
\item\label{formal-decomp}
For every \(I\in\mathcal L(A)\), there exist
\(I_i\in\mathcal L(A_i)\) and \(J\in\mathcal L(C)\) such that
\[
\varphi_1(I_1)=\varphi_2(I_2)=J,
\qquad
I=I_1\times_JI_2.
\]

\item\label{formal-closure}
Whenever
\[
L_i\in\widehat{\mathcal L}(A_i),
\qquad
L_C\in\widehat{\mathcal L}(C),
\]
satisfy
\[
\varphi_1(L_1)=\varphi_2(L_2)=L_C,
\]
the fibre product \(L_1\times_{L_C}L_2\) belongs to
\(\widehat{\mathcal L}(A)\).

\item\label{formal-colonimage}
For every decomposition in \ref{formal-decomp} and every
\(x=(x_1,x_2)\in A^1\), with common image
\(\bar x\in C^1\), one has
\[
\varphi_i\bigl(I_i:(x_i)\bigr)=J:(\bar x),
\qquad i=1,2.
\]
\end{enumerate}
If \(A_1,A_2,C\) are universally Koszul, then \(A\) is universally
Koszul.
\end{proposition}

\begin{proof}
Let \(I\in\mathcal L(A)\) and
\(x=(x_1,x_2)\in A^1\setminus I^1\). By
\ref{formal-decomp}, write
\[
I=I_1\times_JI_2
\]
with \(I_i\in\mathcal L(A_i)\), \(J\in\mathcal L(C)\), and common
image \(J\).

Set
\[
L_i=I_i:(x_i),
\qquad
L_C=J:(\bar x).
\]
Because \(A_i\) is universally Koszul, either
\(x_i\notin I_i\) and Lemma~\ref{lem:colon-criterion} gives
\(L_i\in\mathcal L(A_i)\), or \(x_i\in I_i\) and
\(L_i=A_i\). Hence
\[
L_i\in\widehat{\mathcal L}(A_i).
\]
The same argument gives
\[
L_C\in\widehat{\mathcal L}(C).
\]

By \ref{formal-colonimage} and Lemma~\ref{lem:fibre-colon},
\[
I:(x)=L_1\times_{L_C}L_2.
\]
Hypothesis~\ref{formal-closure} therefore gives
\[
I:(x)\in\widehat{\mathcal L}(A).
\]
But \(I:(x)\neq A\), because
\(1\in I:(x)\) would imply \(x\in I\), contrary to the choice of
\(x\). Thus
\[
I:(x)\in\mathcal L(A).
\]
Lemma~\ref{lem:colon-criterion} now implies that \(A\) is universally
Koszul.
\end{proof}

\begin{remark}
Proposition~\ref{prop:formal-criterion} separates the three ingredients of
the general pullback problem: decomposition of degree-one-generated ideals,
closure of compatible pullbacks, and equality of colon images. The
structural theorems below supply these properties through explicit algebraic
or combinatorial models.
\end{remark}

\section{Structural fibre products}
\label{sec:structural}

\subsection{Orthogonally split fibre products}

The first structural class reduces a fibre product to the direct-sum
operation of Theorem~\ref{thm:directsum-uk}.

\begin{definition}
\label{def:orthsplit}
A surjective graded algebra map
\[
\varphi:A\twoheadrightarrow C
\]
is \emph{orthogonally split} if there exists a connected graded algebra
\(B\) and an isomorphism
\[
A\cong C\sqcap B
\]
under which \(\varphi\) becomes the canonical projection
\[
C\sqcap B\twoheadrightarrow C.
\]
\end{definition}

\begin{theorem}
\label{thm:orthsplit}
Suppose
\[
\varphi_i:A_i\twoheadrightarrow C,\qquad i=1,2,
\]
are orthogonally split, with
\[
A_i\cong C\sqcap B_i.
\]
Then there is a natural graded-algebra isomorphism
\[
A_1\times_C A_2
\cong
C\sqcap B_1\sqcap B_2.
\]
Consequently, if \(C,B_1,B_2\) are universally Koszul, then
\(A_1\times_C A_2\) is universally Koszul.
\end{theorem}

\begin{proof}
Under the given splittings, an element of \(A_i\) can be written
uniquely as
\[
\lambda+c_i+b_i,
\qquad
\lambda\in\Fp,\quad c_i\in C_+,\quad b_i\in(B_i)_+,
\]
and the projection to \(C\) sends it to \(\lambda+c_i\). A pair lies
in the fibre product precisely when the scalar and \(C_+\)-components
agree. Hence every element of the fibre product is uniquely of the
form
\[
\bigl((\lambda+c)+b_1,(\lambda+c)+b_2\bigr).
\]
The map
\[
\bigl((\lambda+c)+b_1,(\lambda+c)+b_2\bigr)
\longmapsto
(\lambda+c)+b_1+b_2
\]
is therefore a graded vector-space isomorphism onto
\(C\sqcap B_1\sqcap B_2\).

Multiplication is preserved because all products between positive
degree elements belonging to distinct direct-sum factors vanish.
Thus the map is a graded-algebra isomorphism. The final assertion
follows by repeated application of
Theorem~\ref{thm:directsum-uk}.
\end{proof}

\begin{remark}
For orthogonally split maps, the fibre product is determined entirely by the
three direct-sum factors. The colon compatibilities of the general criterion
are therefore consequences of the splitting.
\end{remark}

\subsection{Exterior Stanley--Reisner fibre products}

Let \(\Gamma\) be a finite simple graph with vertex set \(V(\Gamma)\).
Define its exterior Stanley--Reisner algebra over \(\Fp\) by
\[
A_\Gamma
=
\Lambda_{\Fp}(x_v:v\in V(\Gamma))
\Big/
\bigl(x_vx_w:\{v,w\}\notin E(\Gamma)\bigr).
\]
A vector-space basis is given by square-free monomials
\[
x_{v_1}\cdots x_{v_r}
\]
whose vertex set is a clique of \(\Gamma\).

If \(\Gamma_0\) is an induced subgraph of \(\Gamma\), killing the
variables outside \(V(\Gamma_0)\) gives a natural surjective map
\[
q_{\Gamma,\Gamma_0}:A_\Gamma\twoheadrightarrow A_{\Gamma_0}.
\]

\begin{proposition}
\label{prop:graph-pullback}
Let \(\Gamma_1,\Gamma_2\) be finite simple graphs with a common induced
subgraph \(\Gamma_0\). Assume
\[
V(\Gamma_1)\cap V(\Gamma_2)=V(\Gamma_0),
\]
and let \(\Gamma\) be the graph with
\[
V(\Gamma)=V(\Gamma_1)\cup V(\Gamma_2),
\qquad
E(\Gamma)=E(\Gamma_1)\cup E(\Gamma_2).
\]
Thus there are no edges joining
\(V(\Gamma_1)\setminus V(\Gamma_0)\) to
\(V(\Gamma_2)\setminus V(\Gamma_0)\). Then
\[
A_\Gamma
\cong
A_{\Gamma_1}\times_{A_{\Gamma_0}}A_{\Gamma_2}
\]
with respect to the natural quotient maps.
\end{proposition}

\begin{proof}
The quotient maps
\[
A_\Gamma\twoheadrightarrow A_{\Gamma_i}
\]
obtained by killing variables outside \(V(\Gamma_i)\) induce a homomorphism
\[
\Psi:A_\Gamma\longrightarrow
A_{\Gamma_1}\times_{A_{\Gamma_0}}A_{\Gamma_2}.
\]
Every clique of \(\Gamma\) lies in \(\Gamma_1\) or in \(\Gamma_2\),
because there are no edges between the two exclusive vertex sets. A clique
lying in both pieces is precisely a clique of \(\Gamma_0\).

Hence a compatible pair in the fibre product has a unique expression in
which the coefficients of the clique monomials supported on \(\Gamma_0\)
agree, while the remaining clique monomials are supported in exactly one of
the two pieces. Sending this compatible pair to the same linear combination
of clique monomials in \(A_\Gamma\) defines an inverse to \(\Psi\).
Therefore \(\Psi\) is a graded-algebra isomorphism.
\end{proof}

\begin{definition}[Diagonal and dominating clique]
\label{def:diagonal-dominating}
A finite graph is \emph{diagonal} if it has no induced subgraph
isomorphic to \(P_4\) or \(C_4\).

Let \(\Gamma_0\) be an induced subgraph of \(\Gamma\). We say that
\(\Gamma_0\) is a \emph{dominating clique} in \(\Gamma\) if
\(\Gamma_0\) is complete and every vertex of \(V(\Gamma_0)\) is
adjacent to every vertex of
\(V(\Gamma)\setminus V(\Gamma_0)\).
\end{definition}

\begin{theorem}
\label{thm:diagonal-gluing}
In the setting of Proposition~\ref{prop:graph-pullback}, assume in
addition that
\begin{enumerate}[label=\textnormal{(\alph*)}]
\item \(\Gamma_1\) and \(\Gamma_2\) are diagonal;
\item \(\Gamma_0\) is a dominating clique in both
\(\Gamma_1\) and \(\Gamma_2\).
\end{enumerate}
Then \(\Gamma\) is diagonal.
\end{theorem}

\begin{proof}
Suppose, toward a contradiction, that \(\Gamma\) contains an induced
subgraph \(\Delta\) isomorphic to \(P_4\) or \(C_4\).

If all vertices of \(\Delta\) lie in one \(\Gamma_i\), then
\(\Gamma_i\) is not diagonal, contrary to hypothesis. Thus
\(\Delta\) contains at least one vertex from
\[
V(\Gamma_1)\setminus V(\Gamma_0)
\]
and at least one vertex from
\[
V(\Gamma_2)\setminus V(\Gamma_0).
\]
There are no edges between these two exclusive sets. Since both
\(P_4\) and \(C_4\) are connected, \(\Delta\) must therefore contain a
vertex \(v\in V(\Gamma_0)\).

By the dominating hypothesis, \(v\) is adjacent to every other vertex
of \(\Delta\): vertices in either exclusive piece are adjacent to
\(v\), and vertices of \(\Gamma_0\) are adjacent to \(v\) because
\(\Gamma_0\) is a clique. Thus \(v\) has degree \(3\) in the induced
four-vertex graph \(\Delta\). This is impossible in \(P_4\) or
\(C_4\), where every vertex has degree at most \(2\). Hence no such
\(\Delta\) exists, and \(\Gamma\) is diagonal.
\end{proof}

\begin{example}[A diagonal gluing without domination]
Let \(\Gamma_0\) consist of one vertex \(v\). Let \(\Gamma_1\) have
vertices \(v,a\) with the edge \(va\), and let \(\Gamma_2\) have vertices
\(v,b\) with no edge. Both pieces are diagonal, while \(\Gamma_0\) is not
dominating in \(\Gamma_2\). Their gluing is an edge together with an
isolated vertex, hence is diagonal.
\end{example}

\begin{example}[A gluing that creates an induced \(P_4\)]
Let \(\Gamma_0\) be the edge on vertices \(v,w\). Let \(\Gamma_1\) be
the path \(a-v-w\), and let \(\Gamma_2\) be the path \(v-w-b\). Each
piece is diagonal, but their gluing is the induced path
\[
a-v-w-b,
\]
which is \(P_4\). Thus preservation of diagonality requires compatibility
between the common subgraph and the two exclusive pieces.
\end{example}

\begin{theorem}[{\cite[Theorem~4.6]{CassellaQuadrelli}}]
\label{thm:CQ}
For a finite graph \(\Gamma\), the exterior Stanley--Reisner algebra
\(A_\Gamma\) is universally Koszul if and only if \(\Gamma\) is
diagonal.
\end{theorem}

\begin{corollary}
\label{cor:graph-uk}
Under the hypotheses of Theorem~\ref{thm:diagonal-gluing},
\[
A_{\Gamma_1}\times_{A_{\Gamma_0}}A_{\Gamma_2}
\]
is universally Koszul.
\end{corollary}

\begin{proof}
By Proposition~\ref{prop:graph-pullback}, the fibre product is
isomorphic to \(A_\Gamma\). By
Theorem~\ref{thm:diagonal-gluing}, \(\Gamma\) is diagonal, and
Theorem~\ref{thm:CQ} gives the conclusion.
\end{proof}

\begin{remark}
Corollary~\ref{cor:graph-uk} expresses universal Koszulity of the pullback in
terms of a combinatorial condition on the glued graph. The defining
quadratic relations are read directly from its nonedges.
\end{remark}

\section{Composita and Kummer gluing}
\label{sec:composita}

We now formulate field composita through the restriction maps supplied by
Galois theory.

\subsection{Contravariance and restriction images}

Let \(k\) be a field, let \(k^{\mathrm{sep}}\) be a fixed separable
closure, and let \(K_1,K_2\) be algebraic subextensions of
\(k^{\mathrm{sep}}/k\). Put
\[
K=K_1K_2.
\]
Then, inside \(G_k=\operatorname{Gal}(k^{\mathrm{sep}}/k)\),
\[
G_{K_i}=\operatorname{Gal}(k^{\mathrm{sep}}/K_i)
\]
and
\[
G_K=\operatorname{Gal}(k^{\mathrm{sep}}/K).
\]

\begin{proposition}
\label{prop:intersection}
With the notation above,
\[
G_K=G_{K_1}\cap G_{K_2}
\]
inside \(G_k\).
\end{proposition}

\begin{proof}
An automorphism of \(k^{\mathrm{sep}}\) fixes \(K_1K_2\) if and only
if it fixes both \(K_1\) and \(K_2\). This is exactly the stated
intersection.
\end{proof}

\begin{remark}
Proposition~\ref{prop:intersection} fixes the group-theoretic direction of
the compositum construction. The cohomological data used below are the
restriction maps from \(K_1\) and \(K_2\) to \(K\).
\end{remark}

From now on assume
\[
\operatorname{char}(k)\neq p,
\qquad
\mu_p\subseteq k.
\]
Fix an identification \(\mu_p\cong\Fp\), and write
\[
H^\bullet(F)
=
H^\bullet(G_F,\Fp)
\]
for absolute Galois cohomology. Standard background on absolute and
maximal pro-\(p\) Galois groups and continuous Galois cohomology may be found
in \cite{NeukirchSchmidtWingberg}.

Kummer theory gives
\[
H^1(F)\cong F^\times/F^{\times p}.
\]
For \(i=1,2\), let
\[
U_i
=
\im\!\left(
\res_{K/K_i}:H^1(K_i)\longrightarrow H^1(K)
\right).
\]
Set
\[
U_0=U_1\cap U_2.
\]

The norm-residue theorem implies that \(H^\bullet(K)\) is generated
in degree one with defining relations generated in degree two
\cite{WeibelNormResidue}. Thus, once
\[
H^1(K)=U_1+U_2,
\]
the quadratic algebra \(H^\bullet(K)\) is determined by the internal
quadratic relations on \(U_1,U_2\), their overlap \(U_0\), and the
mixed products between \(U_1\) and \(U_2\).

Under Kummer theory, if
\[
[a]\in H^1(K_i),
\qquad
[b]\in H^1(K_j),
\]
then after restriction to \(K\) their cup product corresponds to the
Milnor symbol
\[
\{a,b\}\in K_2^M(K)/p.
\]
Consequently the mixed quadratic relations can be tested
arithmetically by symbol vanishing.

\subsection{An orthogonal composita criterion}

We first isolate the simplest situation.

For \(i=1,2\), let \(B_i\subseteq H^\bullet(K)\) be the graded
subalgebra generated by \(U_i\).

\begin{theorem}
\label{thm:orthogonal-composita}
Assume:
\begin{enumerate}[label=\textnormal{(\roman*)}]
\item\label{okc-gen}
\[
H^1(K)=U_1\oplus U_2;
\]

\item\label{okc-mixed}
\[
U_1U_2=U_2U_1=0
\]
in \(H^2(K)\);

\item\label{okc-intersection}
\[
(B_1)_+\cap(B_2)_+=0
\]
inside \(H^\bullet(K)\).
\end{enumerate}
Then the natural map
\[
B_1\sqcap B_2
\longrightarrow
H^\bullet(K)
\]
is an isomorphism of graded algebras. In particular, if \(B_1\) and
\(B_2\) are universally Koszul, then \(H^\bullet(K)\) is universally
Koszul.
\end{theorem}

\begin{proof}
By \ref{okc-gen} and the quadraticity of Galois cohomology,
\(H^\bullet(K)\) is generated by \(U_1\cup U_2\), hence by the two
subalgebras \(B_1,B_2\).

Hypothesis~\ref{okc-mixed} implies that every product involving a
positive-degree element generated from \(U_1\) and a positive-degree
element generated from \(U_2\) vanishes: any mixed monomial contains
an adjacent transition between the two degree-one generating sets,
and the corresponding degree-two product is zero. Thus multiplication
between \((B_1)_+\) and \((B_2)_+\) is zero in either order.

The resulting homomorphism
\[
B_1\sqcap B_2\to H^\bullet(K)
\]
is therefore surjective. Its kernel in positive degree consists
precisely of identifications between positive-degree elements of
\(B_1\) and \(B_2\); by \ref{okc-intersection}, no such nonzero
identification exists. Hence the map is injective.

The universal Koszulity assertion follows from
Theorem~\ref{thm:directsum-uk}.
\end{proof}

\begin{remark}
Hypothesis~\ref{okc-mixed} is equivalent to the vanishing of all mixed
Kummer symbols for classes in the restriction-image subspaces
\(U_1\) and \(U_2\). By bilinearity it is enough to check this on chosen
bases of those subspaces. Together with the positive-degree intersection
condition, it gives the direct-sum presentation of
Theorem~\ref{thm:orthogonal-composita}.
\end{remark}

\subsection{Kummer graph presentations}

The graph criterion applies to square-zero degree-one spaces. For odd \(p\),
degree-one squares vanish by graded commutativity. For \(p=2\), we impose
square vanishing on the degree-one space under consideration. A convenient
field-theoretic sufficient condition is \(-1\in K^{\times2}\), because
\[
\{a,a\}=\{a,-1\}
\]
in \(K_2^M(K)/2\).

\begin{definition}[Kummer graph presentation]
\label{def:kummer-graph}
Let \(B\subseteq H^\bullet(K)\) be a graded subalgebra generated by a
finite-dimensional subspace \(W\subseteq H^1(K)\), and assume
\(x^2=0\) for every \(x\in W\). Let
\[
X=\{x_v:v\in V(\Gamma)\}
\]
be a basis of \(W\) indexed by a finite graph \(\Gamma\). We say that
\(B\) has the \emph{Kummer graph presentation associated with \(\Gamma\)}
if the assignment of the exterior generator \(x_v\) to the corresponding
cohomology class induces an isomorphism
\[
A_\Gamma\overset{\sim}{\longrightarrow}B.
\]
For \(B=H^\bullet(K)\) and \(W=H^1(K)\), this is the Kummer graph
presentation of the full Galois cohomology algebra.
\end{definition}

Because the algebras considered below are quadratic, the presentation is
verified by degree-one generators and the full degree-two relation space. If
\(W\) is square-zero, the cup product factors through
\[
\smile:\bigwedge\nolimits^2 W\longrightarrow H^2(K).
\]
Thus a graph presentation requires equality of the complete kernel of this
map with the Stanley--Reisner relation space; the vanishing pattern of
individual symbols alone does not determine the presentation.

\begin{theorem}
\label{thm:kummer-graph-composita}
Let \(K=K_1K_2\), and assume \(x^2=0\) for every
\(x\in H^1(K)\). Let
\[
U_i=\operatorname{im}(H^1(K_i)\to H^1(K)),
\qquad U_0=U_1\cap U_2.
\]
Suppose there are finite sets
\[
X_0,\quad X_1,\quad X_2
\]
of classes in \(H^1(K)\) satisfying:
\begin{enumerate}[label=\textnormal{(\roman*)}]
\item\label{kg-basis}
\(X_0\) is a basis of \(U_0\), \(X_0\cup X_i\) is a basis of \(U_i\),
and
\[
X_0\cup X_1\cup X_2
\]
is a basis of \(H^1(K)\).

\item\label{kg-piece}
For \(i=1,2\), the subalgebra \(B_i\subseteq H^\bullet(K)\) generated by
\(U_i\) has a Kummer graph presentation
\[
B_i\cong A_{\Gamma_i}
\]
on vertex set \(X_0\cup X_i\). The two induced graphs on \(X_0\)
coincide with a graph \(\Gamma_0\). If \(B_0\) denotes the subalgebra
generated by \(U_0\), assume also that
\[
B_1^2\cap B_2^2=B_0^2
\]
inside \(H^2(K)\).

\item\label{kg-mixed}
Let \(E_{12}\) be a set of edges between \(X_1\) and \(X_2\), put
\[
V_i=\Span_{\Fp}(X_i),
\qquad
M=V_1\wedge V_2,
\]
and define the mixed nonedge relation space
\[
R_{12}
=
\Span_{\Fp}\{x\wedge y:
 x\in X_1,\ y\in X_2,\ \{x,y\}\notin E_{12}\}.
\]
Assume that the cup product on the mixed summand satisfies
\[
\ker\!\left(\smile_K|_M:M\longrightarrow H^2(K)\right)=R_{12}
\]
and that
\[
\smile_K(M)\cap\bigl(B_1^2+B_2^2\bigr)=0.
\]
\end{enumerate}
Let \(\Gamma\) be the graph on \(X_0\cup X_1\cup X_2\) with edge set
\[
E(\Gamma_1)\cup E(\Gamma_2)\cup E_{12}.
\]
Then
\[
H^\bullet(K)\cong A_\Gamma.
\]
If \(\Gamma\) is diagonal, then \(H^\bullet(K)\) is universally Koszul.
\end{theorem}

\begin{proof}
Put
\[
S=\bigwedge\nolimits^2 U_1+\bigwedge\nolimits^2 U_2
\subseteq \bigwedge\nolimits^2 H^1(K).
\]
The basis in Condition~\ref{kg-basis} gives a direct decomposition
\[
\bigwedge\nolimits^2 H^1(K)=S\oplus M.
\]
Let \(R_i\subseteq\bigwedge^2U_i\) be the quadratic relation space of
\(A_{\Gamma_i}\). Condition~\ref{kg-piece} identifies
\(R_i\) with the kernel of the cup product on \(\bigwedge^2U_i\).
We claim that
\[
\ker(\smile_K|_S)=R_1+R_2.
\]
Indeed, let \(z=z_1+z_2\in S\), with
\(z_i\in\bigwedge^2U_i\), and suppose \(\smile_K(z)=0\). Then
\[
\smile_K(z_1)=-\smile_K(z_2)
\in B_1^2\cap B_2^2=B_0^2.
\]
Choose \(z_0\in\bigwedge^2U_0\) representing this common class. Then
\[
z_1-z_0\in R_1,
\qquad
z_2+z_0\in R_2,
\]
so \(z\in R_1+R_2\). The reverse inclusion is immediate.

Now write an arbitrary element of \(\bigwedge^2H^1(K)\) as
\(s+m\), with \(s\in S\) and \(m\in M\), and suppose
\(\smile_K(s+m)=0\). Since \(\smile_K(s)\in B_1^2+B_2^2\), we have
\[
\smile_K(m)=-\smile_K(s)
\in \smile_K(M)\cap(B_1^2+B_2^2)=0.
\]
Hence \(m\in R_{12}\), while \(s\in R_1+R_2\). Therefore
\[
\ker\!\left(
\smile_K:\bigwedge\nolimits^2H^1(K)\longrightarrow H^2(K)
\right)
=R_1+R_2+R_{12},
\]
which is exactly the Stanley--Reisner quadratic relation space
\(R_\Gamma\).

By Condition~\ref{kg-basis}, the vertices of \(\Gamma\) give all
of \(H^1(K)\). The norm-residue theorem makes \(H^\bullet(K)\)
quadratic, so the equality of degree-one generators and quadratic
relation spaces yields
\[
A_\Gamma\overset{\sim}{\longrightarrow}H^\bullet(K).
\]
If \(\Gamma\) is diagonal, Theorem~\ref{thm:CQ} gives universal
Koszulity.
\end{proof}

\begin{remark}
\label{rem:kummer-arithmetic}
Condition~\ref{kg-basis} is computed from the maps
\[
K_i^\times/K_i^{\times p}\longrightarrow K^\times/K^{\times p}.
\]
Condition~\ref{kg-piece} records the internal relation spaces of the two
restriction-image subalgebras and their degree-two overlap. Condition~\ref{kg-mixed} then determines the genuinely mixed relations: the mixed
nonedges are exactly the kernel of the mixed cup-product map, while the
second equality excludes additional linear relations between mixed cup
products and the internal degree-two classes. Under Kummer theory these
conditions are expressed by Milnor symbols
\[
\{a,b\}\in K_2^M(K)/p.
\]
Thus the graph is determined by explicit degree-one and degree-two arithmetic
data.
\end{remark}

\subsection{A verified family: monomial Laurent-series composita}
\label{subsec:laurent-composita}

We now verify the preceding criterion for an explicit infinite family.
The family exhibits nonzero mixed cup products.

\begin{lemma}
\label{lem:laurent-cohomology}
Let \(\kappa\) be an algebraically closed field with
\(\operatorname{char}(\kappa)\neq p\), and put
\[
F_d=\kappa((t_1))\cdots((t_d)).
\]
Let \(e_i\in H^1(F_d,\Fp)\) be the Kummer class of \(t_i\).
Then
\[
H^\bullet(F_d,\Fp)
\cong
\Lambda_{\Fp}(e_1,\dots,e_d).
\]
In particular, the products
\[
e_{i_1}\cdots e_{i_r},
\qquad
1\leq i_1<\cdots<i_r\leq d,
\]
form an \(\Fp\)-basis of \(H^r(F_d,\Fp)\).
\end{lemma}

\begin{proof}
We argue by induction on \(d\). For \(d=0\), the assertion follows
because \(\kappa\) is algebraically closed.

Let \(F_d=F_{d-1}((t_d))\). The Witt exact sequence for a complete
discretely valued field with residue field \(F_{d-1}\)
\cite[Proposition~7.7]{GaribaldiMerkurjevSerre} gives, after choosing
the uniformizer \(t_d\), a split decomposition
\[
H^q(F_d,\mu_p^{\otimes q})
\cong
H^q(F_{d-1},\mu_p^{\otimes q})
\oplus
H^{q-1}(F_{d-1},\mu_p^{\otimes(q-1)})\cup(t_d).
\]
Since \(\kappa\) is algebraically closed, it contains \(\mu_p\), so
after fixing the usual identification of twists this is the
decomposition of \(H^q(F_d,\Fp)\).

The new degree-one class \(e_d=(t_d)\) satisfies \(e_d^2=0\).
For odd \(p\) this follows from graded commutativity. For \(p=2\),
the Milnor relation \(\{a,a\}=\{a,-1\}\) and the fact that
\(-1\in\kappa^{\times2}\) give the same conclusion. The split Witt
decomposition is therefore multiplicatively
\[
H^\bullet(F_d,\Fp)
\cong
H^\bullet(F_{d-1},\Fp)\otimes_{\Fp}\Lambda_{\Fp}(e_d).
\]
The induction hypothesis completes the proof.
\end{proof}

Fix the iterated Laurent-series tower
\[
M_0=\kappa,
\qquad
M_j=M_{j-1}((u_j)),
\qquad
M=M_d.
\]
For \(S\subseteq[d]=\{1,\dots,d\}\), define recursively
\[
K_{S,0}=\kappa,
\qquad
K_{S,j}=K_{S,j-1}((v_{S,j}))\subseteq M_j,
\]
where
\[
v_{S,j}
=
\begin{cases}
u_j,&j\in S,\\
u_j^p,&j\notin S.
\end{cases}
\]
Set \(K_S=K_{S,d}\). This recursive construction fixes the order of the
valuations and gives a specified embedding of every \(K_S\) in the common
ambient field \(M\).

\begin{theorem}
\label{thm:laurent-composita}
Let \(S,T\subseteq[d]\). Then:
\begin{enumerate}[label=\textnormal{(\alph*)}]
\item\label{laurent-fields}
\[
K_SK_T=K_{S\cup T},
\qquad
K_S\cap K_T=K_{S\cap T}.
\]

\item\label{laurent-cohomology}
For every \(R\subseteq[d]\),
\[
H^\bullet(K_R,\Fp)
\cong
\Lambda_{\Fp}(e_{R,1},\dots,e_{R,d}),
\]
where \(e_{R,j}\) is the Kummer class of \(v_{R,j}\).

\item\label{laurent-images}
Put \(K=K_{S\cup T}\) and set
\[
X_1=S\setminus T,\qquad
X_2=T\setminus S,\qquad
X_0=(S\cap T)\cup\bigl([d]\setminus(S\cup T)\bigr).
\]
If \(e_j=e_{S\cup T,j}\), then
\[
U_1
=
\Span_{\Fp}\{e_j:j\in X_0\cup X_1\},
\qquad
U_2
=
\Span_{\Fp}\{e_j:j\in X_0\cup X_2\},
\]
and
\[
U_1\cap U_2
=
\Span_{\Fp}\{e_j:j\in X_0\}.
\]

\item\label{laurent-graph}
The Kummer graph of \(H^\bullet(K,\Fp)\) with respect to
\(\{e_1,\dots,e_d\}\) is the complete graph \(K_d\). In particular,
all mixed products
\[
e_i e_j,
\qquad
i\in X_1,\quad j\in X_2,
\]
are nonzero, and \(H^\bullet(K,\Fp)\) is universally Koszul.
\end{enumerate}
Thus, whenever both \(S\setminus T\) and \(T\setminus S\) are
nonempty, \(K_S\) and \(K_T\) are proper incomparable subfields of
their compositum and give a nontrivial verified composita family.
\end{theorem}

\begin{proof}
Part~\ref{laurent-cohomology} follows from
Lemma~\ref{lem:laurent-cohomology}, with parameters
\(v_{R,1},\dots,v_{R,d}\).

Let \(I=S\cap T\) and \(R=S\cup T\). The extension
\[
K_R/K_I
\]
is obtained by adjoining \(p\)-th roots of the parameters
\(u_j^p\), for \(j\in R\setminus I\). By
Part~\ref{laurent-cohomology}, their Kummer classes in
\(H^1(K_I,\Fp)\) are linearly independent. Kummer theory therefore
gives
\[
\operatorname{Gal}(K_R/K_I)
\cong
(\mathbb Z/p\mathbb Z)^{\,|R\setminus I|}.
\]
The classes indexed by \(R\setminus I\) are independent in
\(K_I^\times/K_I^{\times p}\), so the corresponding cyclic Kummer
extensions are linearly disjoint over \(K_I\). The extension \(K_S/K_I\)
is generated by the roots indexed by \(S\setminus I\), and \(K_T/K_I\) is
generated by the disjoint set indexed by \(T\setminus I\). Under
\[
\operatorname{Gal}(K_R/K_I)\cong(\mathbb Z/p\mathbb Z)^{\,|R\setminus I|},
\]
their fixed subgroups are the coordinate subgroups complementary to these
two sets. The elementary Galois correspondence therefore gives
\[
K_SK_T=K_R,
\qquad
K_S\cap K_T=K_I,
\]
proving Part~\ref{laurent-fields}.

We next compute restriction on \(H^1\). A parameter of \(K_S\)
indexed by \(j\in T\setminus S\) is \(u_j^p\), which becomes the
\(p\)-th power of \(u_j\) in \(K_R\); its Kummer class therefore
restricts to zero. Every other parameter of \(K_S\) is exactly the
corresponding parameter of \(K_R\), so its class restricts to \(e_j\).
This proves the formula for \(U_1\), and the formula for \(U_2\) is
symmetric. Their intersection is then the span indexed by \(X_0\),
proving Part~\ref{laurent-images}.

Finally, Part~\ref{laurent-cohomology} identifies
\(H^\bullet(K,\Fp)\) with the exterior algebra on all \(d\) classes.
Equivalently, its Kummer graph is the complete graph \(K_d\). Hence
every wedge \(e_ie_j\) with \(i\neq j\), including every mixed product
with \(i\in X_1\) and \(j\in X_2\), is nonzero. The complete graph is
diagonal, so Theorem~\ref{thm:CQ} yields universal Koszulity.
\end{proof}

\begin{example}[The two-parameter compositum]
\label{ex:two-parameter}
Take \(d=2\), \(S=\{1\}\), and \(T=\{2\}\). Then
\[
K_0=\kappa((u_1^p))((u_2^p)),
\]
\[
K_1=\kappa((u_1))((u_2^p)),
\qquad
K_2=\kappa((u_1^p))((u_2)),
\]
and
\[
K=\kappa((u_1))((u_2)).
\]
Theorem~\ref{thm:laurent-composita} gives
\[
K=K_1K_2,
\qquad
K_1\cap K_2=K_0.
\]
If \(e_1=(u_1)\) and \(e_2=(u_2)\), then
\[
U_1=\Fp e_1,\qquad
U_2=\Fp e_2,
\qquad
e_1e_2\neq0.
\]
Thus
\[
H^\bullet(K,\Fp)
\cong
\Lambda_{\Fp}(e_1,e_2),
\]
so the compositum is governed by a complete two-vertex Kummer graph. The
nonzero mixed symbol \(e_1e_2\) is the simplest instance of the mixed-edge
part of Theorem~\ref{thm:kummer-graph-composita}.
\end{example}

\section{Nonabelian composita from diagonal pro-\texorpdfstring{\(p\)}{p} RAAGs}
\label{sec:raag-composita}

We now treat a nonabelian family in which the maximal pro-\(p\) Galois
groups are right-angled Artin pro-\(p\) groups and the cohomology algebras
have nontrivial zero-product relations.

We use the following elementary fact about finite subextensions of a
maximal pro-\(p\) extension.

\begin{lemma}
\label{lem:persistence}
Let \(k\) be a field and let \(L/k\) be a finite extension with
\[
k\subseteq L\subseteq k(p).
\]
Then
\[
L(p)=k(p).
\]
Consequently,
\[
G_L(p)=\operatorname{Gal}(k(p)/L)
\]
is an open subgroup of \(G_k(p)\).
\end{lemma}

\begin{proof}
Since \(k(p)/L\) is Galois with pro-\(p\) Galois group, it is the union
of its finite Galois \(p\)-subextensions over \(L\). Hence
\(k(p)\subseteq L(p)\).

Conversely, let \(M/L\) be a finite Galois \(p\)-extension. Let
\(L^{\mathrm n}/k\) be the normal closure of \(L/k\) inside \(k(p)\).
Because \(k(p)/k\) is pro-\(p\), the finite Galois group
\(\operatorname{Gal}(L^{\mathrm n}/k)\) is a \(p\)-group. For every
\(k\)-embedding of \(M\) into a separable closure, the compositum of
the conjugate of \(M\) with \(L^{\mathrm n}\) is a finite Galois
\(p\)-extension of \(L^{\mathrm n}\). The compositum \(N\) of these
finitely many conjugates is therefore Galois over \(k\), and
\[
1\longrightarrow \operatorname{Gal}(N/L^{\mathrm n})
\longrightarrow \operatorname{Gal}(N/k)
\longrightarrow \operatorname{Gal}(L^{\mathrm n}/k)
\longrightarrow 1
\]
has \(p\)-groups at both ends. Thus \(N/k\) is a finite Galois
\(p\)-extension, so \(N\subseteq k(p)\). In particular
\(M\subseteq k(p)\). Therefore \(L(p)\subseteq k(p)\), proving the
claim.
\end{proof}

\begin{theorem}
\label{thm:raag-lattice}
Let \(p\) be a prime and let \(k\) be a field containing \(\mu_p\).
Assume that
\[
G_k(p)\cong G_\Gamma
\]
for the pro-\(p\) right-angled Artin group associated to a finite
diagonal graph \(\Gamma\). Then every finite extension
\[
k\subseteq L\subseteq k(p)
\]
has universally Koszul mod-\(p\) Galois cohomology.

Consequently, if \(K_1,\dots,K_m\) are finite subextensions of
\(k(p)/k\), then every field obtained from them by finitely many
composita and intersections has universally Koszul mod-\(p\) Galois
cohomology.
\end{theorem}

\begin{proof}
By Lemma~\ref{lem:persistence},
\[
G_L(p)=\operatorname{Gal}(k(p)/L)
\]
is an open, hence finitely generated, subgroup of \(G_\Gamma\).
Snopce--Zalesskii's Theorem~1.2 gives the equivalence between diagonalness
of \(\Gamma\) and the hereditary property that every closed subgroup of
\(G_\Gamma\) is a pro-\(p\) RAAG. Applying this implication to \(G_L(p)\)
gives
\[
G_L(p)\cong G_\Lambda
\]
for a graph \(\Lambda\). Since \(G_L(p)\) is finitely generated,
\(\Lambda\) has finitely many vertices.

The same theorem separately gives the equivalence between diagonalness of a
finite graph and realizability of its pro-\(p\) RAAG as \(G_F(p)\) for a
field \(F\) containing a primitive \(p\)-th root of unity. Since
\(G_\Lambda\cong G_L(p)\), the graph \(\Lambda\) is diagonal.
Cassella--Quadrelli's criterion then yields universal Koszulity of
\[
H^\bullet(G_L(p),\Fp).
\]

A compositum or intersection of finitely many finite subextensions of
\(k(p)/k\) remains a finite subextension of \(k(p)/k\), giving the final
assertion.
\end{proof}

\begin{remark}
Under the Galois correspondence, composita correspond to intersections of
open subgroups. The hereditary RAAG property therefore gives closure of the
finite \(p\)-extension lattice.
\end{remark}

For every prime \(p\), including \(p=2\), the continuous mod-\(p\)
cohomology of a pro-\(p\) RAAG has its standard exterior
Stanley--Reisner presentation. Thus square vanishing in degree one is part
of the RAAG cohomology presentation itself.

We now make this construction explicit on an infinite nonabelian
family. Let \(\mathcal E_d\) denote the edgeless graph on \(d\)
vertices and let \(K_s\) denote the complete graph on \(s\) vertices.
The join
\[
\Gamma_{d,s}=\mathcal E_d\ast K_s
\]
is diagonal: every vertex coming from \(K_s\) is dominating, while the
vertices coming from \(\mathcal E_d\) have no edges among themselves.
Its pro-\(p\) RAAG is
\[
G_{\Gamma_{d,s}}\cong F_d\times \mathbb Z_p^s,
\]
where \(F_d\) is the free pro-\(p\) group of rank \(d\).
By Snopce--Zalesskii, there exists a field \(k\) containing \(\mu_p\)
with this maximal pro-\(p\) Galois group.

\begin{theorem}
\label{thm:split-raag-lattice}
Let \(k\) contain \(\mu_p\) and suppose
\[
G_k(p)\cong F_d\times\mathbb Z_p^s,
\qquad d\ge2,\quad s\ge1.
\]
For every finite extension
\[
k\subseteq L\subseteq k(p),
\]
put \(H=G_L(p)\), viewed as an open subgroup of
\(F_d\times\mathbb Z_p^s\), and let
\[
P=\operatorname{pr}_{F_d}(H)\le F_d.
\]
Then \(P\) is open in \(F_d\), and if
\[
q=[F_d:P],
\]
then
\[
G_L(p)\cong F_{\,1+q(d-1)}\times\mathbb Z_p^s.
\]
Consequently
\[
H^\bullet(G_L(p),\Fp)
\cong
A_{\mathcal E_{1+q(d-1)}\ast K_s},
\]
so every field in the finite \(p\)-extension lattice of \(k\) has an
explicit diagonal Stanley--Reisner cohomology presentation.
\end{theorem}

\begin{proof}
By Lemma~\ref{lem:persistence}, \(H=G_L(p)\) is an open subgroup of
\(F_d\times\mathbb Z_p^s\). Write
\[
Z=\mathbb Z_p^s.
\]
Since \(H\) is open, \(H\cap Z\) is open in \(Z\). Every open subgroup
of \(\mathbb Z_p^s\) is abstractly isomorphic to \(\mathbb Z_p^s\), so
\[
H\cap Z\cong\mathbb Z_p^s.
\]
We use this abstract pro-\(p\) group isomorphism in the decomposition below.
Moreover, the image
\[
P=\operatorname{pr}_{F_d}(H)
\]
is open in \(F_d\). Projection gives an exact sequence
\[
1\longrightarrow H\cap Z
\longrightarrow H
\longrightarrow P
\longrightarrow1.
\]
The group \(P\) is an open subgroup of a free pro-\(p\) group and is
therefore free pro-\(p\). Free pro-\(p\) groups are projective in the
category of pro-\(p\) groups, so the displayed sequence splits. Since
\(H\cap Z\) lies in the central factor \(Z\), it is central in \(H\);
therefore the splitting is a direct product:
\[
H\cong P\times(H\cap Z).
\]
The pro-\(p\) Schreier formula gives
\[
P\cong F_{1+q(d-1)},
\qquad q=[F_d:P],
\]
and hence
\[
G_L(p)=H\cong F_{1+q(d-1)}\times\mathbb Z_p^s.
\]
The final statement follows from the standard RAAG cohomology
presentation and the fact that
\(\mathcal E_{1+q(d-1)}\ast K_s\) is diagonal.
\end{proof}

\begin{remark}
Theorem~\ref{thm:split-raag-lattice} gives more than preservation of
universal Koszulity: it determines the isomorphism type of the maximal
pro-\(p\) Galois group throughout the finite \(p\)-extension lattice inside \(k(p)\),
up to the single numerical invariant
\([F_d:\operatorname{pr}_{F_d}G_L(p)]\). In particular, composita and
field intersections stay inside the same split-RAAG family.
\end{remark}

\begin{theorem}
\label{thm:split-raag-composita}
Fix integers \(d\ge2\), \(s\ge1\), and \(1\le r\le d\). Let \(k\) be
a field containing \(\mu_p\) such that
\[
G_k(p)\cong F_d\times\mathbb Z_p^s.
\]
Choose linearly independent characters
\[
\chi_1,\dots,\chi_r\in H^1(F_d,\Fp).
\]
Via pullback along the projection
\[
\operatorname{pr}_{F_d}:F_d\times\mathbb Z_p^s\longrightarrow F_d,
\]
regard them as classes in \(H^1(F_d\times\mathbb Z_p^s,\Fp)\).
Equivalently, they are trivial on the factor \(\mathbb Z_p^s\). By
Kummer theory choose classes \(a_i\in k^\times/k^{\times p}\)
corresponding to these characters, and put
\[
K_i=k(\sqrt[p]{a_i}),
\qquad
K=K_1\cdots K_r.
\]
Then the fields \(K_i/k\) are linearly disjoint cyclic extensions of
degree \(p\),
\[
\operatorname{Gal}(K/k)\cong (\mathbb Z/p\mathbb Z)^r,
\]
and
\[
G_K(p)
\cong
F_{\,1+p^r(d-1)}\times\mathbb Z_p^s.
\]
In particular,
\[
H^\bullet(G_K(p),\Fp)
\cong
A_{\mathcal E_{1+p^r(d-1)}\ast K_s}
\]
and this algebra is universally Koszul.
\end{theorem}

\begin{proof}
Let
\[
\chi=(\chi_1,\dots,\chi_r):
F_d\times\mathbb Z_p^s\longrightarrow \Fp^r.
\]
The characters are independent and are trivial on the second factor,
so \(\chi\) is surjective and
\[
\ker(\chi)
=
\ker(\chi|_{F_d})\times\mathbb Z_p^s.
\]
The fixed field of \(\ker(\chi_i)\) is \(K_i\), and the fixed field of
\(\bigcap_i\ker(\chi_i)=\ker(\chi)\) is their compositum \(K\).
Therefore the \(K_i\) are linearly disjoint over \(k\) and
\(\operatorname{Gal}(K/k)\cong\Fp^r\).

By Lemma~\ref{lem:persistence},
\[
G_K(p)=\ker(\chi).
\]
The subgroup \(\ker(\chi|_{F_d})\) has index \(p^r\) in the free
pro-\(p\) group \(F_d\). The pro-\(p\) Schreier formula gives
\[
d\bigl(\ker(\chi|_{F_d})\bigr)
=1+p^r(d-1),
\]
so
\[
\ker(\chi|_{F_d})\cong F_{1+p^r(d-1)}.
\]
This proves the asserted group isomorphism.

The group \(F_n\times\mathbb Z_p^s\), with
\(n=1+p^r(d-1)\), is the pro-\(p\) RAAG attached to the join
\(\mathcal E_n\ast K_s\). Hence its cohomology is the exterior
Stanley--Reisner algebra of that graph. The graph is diagonal for the
same reason as \(\Gamma_{d,s}\), and universal Koszulity follows from
Theorem~\ref{thm:CQ}.
\end{proof}

\begin{corollary}
\label{cor:minimal-raag-compositum}
Let \(k\) contain \(\mu_p\) and satisfy
\[
G_k(p)\cong F_2\times\mathbb Z_p.
\]
There exist linearly disjoint cyclic degree-\(p\) extensions
\(K_1/k\) and \(K_2/k\) such that, for
\(K=K_1K_2\),
\[
G_K(p)\cong F_{p^2+1}\times\mathbb Z_p.
\]
Thus
\[
H^\bullet(G_K(p),\Fp)
\cong
A_{\mathcal E_{p^2+1}\ast K_1},
\]
the exterior Stanley--Reisner algebra of a star with \(p^2+1\)
leaves. In particular it is universally Koszul.
\end{corollary}

\begin{proof}
Apply Theorem~\ref{thm:split-raag-composita} with
\(d=2\), \(s=1\), and \(r=2\).
\end{proof}

\section{Further directions}
\label{sec:further}

The preceding results suggest three directions in which the gluing methods can
be developed further.

\paragraph{Graph and algebraic gluing.}
The dominating-clique hypothesis in
Theorem~\ref{thm:diagonal-gluing} is sufficient but not necessary, as
Example~4.7 shows. A natural graph-theoretic problem is to characterize
exactly which gluings along a common induced subgraph preserve diagonality.
On the algebraic side, one can seek further structural decompositions of
fibre products for which the ideal and colon compatibilities of
Proposition~\ref{prop:formal-criterion} follow from the form of the
decomposition itself.

\paragraph{Arithmetic Kummer gluing.}
For a compositum \(K=K_1K_2\), the basic arithmetic data are the span
\[
U_1+U_2\subseteq H^1(K)
\]
and the mixed symbol map
\[
U_1\otimes U_2\longrightarrow H^2(K).
\]
Theorem~\ref{thm:laurent-composita} gives an explicit family in which both
are computed exactly and the mixed symbols are nonzero.
Theorem~\ref{thm:kummer-graph-composita} shows how the internal relation
spaces, their degree-two overlap, and the mixed symbol map assemble into the
quadratic presentation of the compositum. Computing these data for further
arithmetic families would produce new classes of composita whose universal
Koszulity is decided directly from Kummer theory.

\paragraph{RAAG and \(\Delta\)-RAAG composita.}
Theorem~\ref{thm:split-raag-composita} computes the maximal pro-\(p\) Galois
group when the Kummer characters are supported on the free factor of
\(F_d\times\mathbb Z_p^s\). Allowing characters with nontrivial components
on both factors leads to the problem of identifying the resulting RAAG graph
directly from the Kummer data.

For \(p=2\), Hamza, Maire, Min\'{a}\v{c}, and T\^{a}n
\cite{HamzaMaireMinacTan} describe maximal pro-\(2\) Galois groups of
formally real Pythagorean fields of finite type using
\(\Delta\)-right-angled Artin pro-\(2\) groups. Their graph data provide a
natural setting for the same composita calculation: one seeks the
restriction-image basis in degree one and the complete relation space
\[
\ker\left(\bigwedge\nolimits^2H^1(K,\Ftwo)\longrightarrow H^2(K,\Ftwo)\right).
\]
Identifying these data for explicit \(\Delta\)-RAAG composita would connect
the Kummer-gluing criterion with this broader graph-theoretic Galois
framework.


\begin{thebibliography}{99}

\bibitem{CassellaQuadrelli}
A.~Cassella and C.~Quadrelli,
``Right-angled Artin groups and enhanced Koszul properties,''
\textit{J. Group Theory} \textbf{24} (2021), no.~1, 17--38.

\bibitem{ConcaUniversallyKoszul}
A.~Conca,
``Universally Koszul algebras,''
\textit{Math. Ann.} \textbf{317} (2000), 329--346.

\bibitem{GaribaldiMerkurjevSerre}
S.~Garibaldi, A.~Merkurjev, and J.-P.~Serre,
\textit{Cohomological Invariants in Galois Cohomology},
University Lecture Series \textbf{28},
American Mathematical Society, Providence, RI, 2003.

\bibitem{HamzaMaireMinacTan}
O.~Hamza, C.~Maire, J.~Min\'{a}\v{c}, and N.~D.~T\^{a}n,
``Maximal \(2\)-extensions of Pythagorean fields and
right-angled Artin groups,''
arXiv:2510.11970 (2025).

\bibitem{MinacPalaistiPasiniTan}
J.~Min\'{a}\v{c}, M.~Palaisti, F.~W.~Pasini, and N.~D.~T\^{a}n,
``Enhanced Koszul properties in Galois cohomology,''
\textit{Res. Math. Sci.} \textbf{7} (2020), Paper~10.

\bibitem{MinacPasiniQuadrelliTan}
J.~Min\'{a}\v{c}, F.~W.~Pasini, C.~Quadrelli, and N.~D.~T\^{a}n,
``Koszul algebras and quadratic duals in Galois cohomology,''
\textit{Adv. Math.} \textbf{380} (2021), Paper~107569.

\bibitem{NeukirchSchmidtWingberg}
J.~Neukirch, A.~Schmidt, and K.~Wingberg,
\textit{Cohomology of Number Fields},
2nd ed., Grundlehren der Mathematischen Wissenschaften
\textbf{323}, Springer, 2008.

\bibitem{PalaistiThesis}
M.~Palaisti,
\textit{Enhanced Koszulity in Galois Cohomology},
Ph.D. thesis, The University of Western Ontario, London, Ontario, 2019.

\bibitem{SnopceZalesskii}
I.~Snopce and P.~A.~Zalesskii,
``Right-angled Artin pro-\(p\) groups,''
\textit{Bull. Lond. Math. Soc.} \textbf{54} (2022), no.~5,
1904--1922.

\bibitem{WeibelNormResidue}
C.~A.~Weibel,
``The norm residue isomorphism theorem,''
\textit{J. Topol.} \textbf{2} (2009), no.~2, 346--372.

\end{thebibliography}
\end{document}